\documentclass{amsart}

\usepackage{amssymb, latexsym, amsmath, amscd, amsthm, mathrsfs}
\usepackage[pdfborder={0 0 0},backref=false,colorlinks=true,linktocpage=true]{hyperref}

\newtheorem{theorem}{Theorem}[section]
\newtheorem{lemma}[theorem]{Lemma}
\newtheorem{definition}[theorem]{Definition}

\newtheorem{proposition}[theorem]{Proposition}

\newtheorem{corollary}[theorem]{Corollary}

\begin{document}

\title{Irreducibles and Primes in Computable Integral Domains}

\author[Evron]{Leigh Evron}
\address{Department of Mathematics and Statistics\\
Grinnell College\\
Grinnell, Iowa 50112 U.S.A.}
\email{evronlei@grinnell.edu}

\author[Mileti]{Joseph R. Mileti}
\address{Department of Mathematics and Statistics\\
Grinnell College\\
Grinnell, Iowa 50112 U.S.A.}
\email{miletijo@grinnell.edu}

\author[Ratliff-Crain]{Ethan Ratliff-Crain}
\address{Department of Mathematics and Statistics\\
Grinnell College\\
Grinnell, Iowa 50112 U.S.A.}
\email{ratliffc@grinnell.edu}

\thanks{The authors thank Grinnell College for its generous support through the MAP program for research with undergraduates.}

\begin{abstract}
A computable ring is a ring equipped with mechanical procedure to add and multiply elements.  In most natural computable integral domains, there is a computational procedure to determine if a given element is prime/irreducible.  However, there do exist computable UFDs (in fact, polynomial rings over computable fields) where the set of prime/irreducible elements is not computable.  Outside of the class of UFDs, the notions of irreducible and prime may not coincide.  We demonstrate how different these concepts can be by constructing computable integral domains where the set of irreducible elements is computable while the set of prime elements is not, and vice versa.  Along the way, we will generalize Kronecker's method for computing irreducibles and factorizations in $\mathbb{Z}[x]$.
\end{abstract}

\maketitle

\section{Introduction}

In an integral domain, there are two natural definitions of basic ``atomic" elements:  irreducibles and primes.  We recall these standard algebraic definitions.

\begin{definition}
Let $A$ be an integral domain, i.e.~a commutative ring with $1 \neq 0$ and with no zero divisors (so $ab = 0$ implies either $a = 0$ or $b=0$).  Recall the following definitions.
\begin{enumerate}
\item An element $u \in A$ is a {\em unit} if there exists $w \in A$ with $uw = 1$.  We denote the set of units by $U(A)$.  Notice that $U(A)$ is a multiplicative group.
\item Given $a,b \in A$, we say that $a$ and $b$ are {\em associates} if there exists $u \in U(A)$ with $au = b$.
\item An element $p \in A$ is {\em irreducible} if it nonzero, not a unit, and has the property that whenever $p = ab$, either $a$ is a unit or $b$ is a unit.  An equivalent definition is that $p \in A$ is irreducible if it is nonzero, not a unit, and its divisors are precisely the units and the associates of $p$.
\item An element $p \in A$ is {\em prime} if it nonzero, not a unit, and has the property that whenever $p \mid ab$, either $p \mid a$ or $p \mid b$.
\item $A$ is a {\em unique factorization domain}, or {\em UFD}, if it has the following two properties:
\begin{itemize}
\item For each $a \in A$ such that $a$ is nonzero and not a unit, there exist irreducible elements $r_1,r_2,\dots,r_n \in A$ with $a = r_1r_2 \cdots r_n$.
\item If $r_1,r_2,\dots,r_n,q_1,q_2,\dots,q_m \in A$ are all irreducible and $r_1r_2 \cdots r_n = q_1q_2 \cdots q_m$, then $n = m$ and there exists a permutation $\sigma$ of $\{1,2,\dots,n\}$ such that $r_i$ and $q_{\sigma(i)}$ are associates for all $i$.
\end{itemize}
\end{enumerate}
\end{definition}

It is a simple fact that if $A$ is an integral domain, then every prime element of $A$ is irreducible.  Although the converse is true in any UFD, it does fail for general integral domains.  For example, in the integral domain $\mathbb{Z}[\sqrt{-5}]$, there are two different factorizations of $6$ into irreducibles:
\[
2 \cdot 3 = 6 = (1 + \sqrt{-5})(1 - \sqrt{-5}).
\]
Since $U(\mathbb{Z}(\sqrt{-5})) = \{1,-1\}$, these two factorizations are indeed distinct.  This example also shows that $2$ is an irreducible element that is not prime because $2 \mid (1 + \sqrt{-5})(1 - \sqrt{-5})$ but $2 \nmid 1 + \sqrt{-5}$ and $2 \nmid 1 - \sqrt{-5}$.  In fact, all four of the above irreducible factors are not prime.

For another example that will be particularly relevant for our purposes, let $A$ be the subring of $\mathbb{Q}[x]$ consisting of those polynomials whose constant term and coefficient of $x$ are both integers, i.e.
\[
A = \{a_0 + a_1x + a_2x^2 + \dots + a_nx^n \in \mathbb{Q}[x] : a_0 \in \mathbb{Z} \text{ and } a_1 \in \mathbb{Z}\}.
\]
In this integral domain, all of the normal integer primes are still irreducible (by a simple degree argument), but none of them are prime in $A$ because given any integer prime $p \in \mathbb{Z}$, we have that $p \mid x^2$ since $\frac{x^2}{p} \in A$, but $p \nmid x$ as $\frac{x}{p} \notin A$.

We are interested in the extent to which the irreducible and prime elements can differ in an integral domain.  As just discussed, the set of prime elements is always a subset of the set of irreducible elements, but it may be a proper subset.  Can one of these sets be significantly more complicated than the other?  We approach this question from the point of view of computability theory.  We begin with the following fundamental definition.

\begin{definition}
A {\em computable ring} is a ring whose underlying set is a computable set $A \subseteq \mathbb{N}$, with the property that $+$ and $\cdot$ are computable functions from $A \times A$ to $A$.
\end{definition}

For a general overview of results about computable rings and fields, see \cite{SHTucker}.  Computable fields together with computable factorizations in polynomial rings over those fields have received a great deal of attention (\cite{FrohlichShep}, \cite{MetakidesNerode}, \cite{Rabin}), and \cite{MillerNotices} provides an excellent overview of work in this area.  In particular, there exists a computable field $F$ such that the set of primes in $F[x]$ is not computable (see \cite[Lemma 3.4]{MillerNotices} or \cite[Section 3.2]{SHTucker} for an example).  Moreover, there is a computable UFD such that the set of primes is as complicated as possible in the arithmetical hierarchy (see \cite{JoeDamir}).  For our purposes, we will only need the first level of this hierarchy (see \cite[Chapter 4]{Soare} for more information).

\begin{definition}
Let $Z \subseteq \mathbb{N}$.
\begin{itemize}
\item We say that $Z$ is a $\Sigma_1^0$ set, or {\em computably enumerable}, if there exists a computable $R \subseteq \mathbb{N}^2$ such that
\[
i \in Z \Longleftrightarrow (\exists x) R(x,i).
\]
\item We say that $Z$ is a $\Pi_1^0$ set if there exists a computable $R \subseteq \mathbb{N}^2$ such that
\[
i \in Z \Longleftrightarrow (\forall x) R(x,i).
\]
\end{itemize}
\end{definition}

Notice that the complement of $\Sigma_1^0$ set is a $\Pi_1^0$ set, and the complement of $\Pi_1^0$ set is a $\Sigma_1^0$ set.  Although every computable set is both a $\Sigma_1^0$ set and $\Pi_1^0$ set, there exists a $\Sigma_1^0$ set that is not computable, such as the set of natural numbers coding programs that halt.  The complement of a noncomputable $\Sigma_1^0$ set is a noncomputable $\Pi_1^0$ set.  We will use the following standard fact (see \cite[Section II.1]{Soare})

\begin{proposition} \label{p:Sigma1IffRangeComputable}
An infinite set $Z \subseteq \mathbb{N}$ is $\Sigma_1^0$ if and only if there exists a computable injective function $\alpha \colon \mathbb{N} \to \mathbb{N}$ such that $\text{range}(\alpha) = Z$.
\end{proposition}

We will prove that there exists a computable integral domain where the set of irreducible elements is computable while the set of prime elements is not, and also there exists a computable integral domain where the set of prime elements is computable while the set of irreducible elements is not.  Thus, these two notions can be wildly different.  Our approach will be to code an arbitrary $\Pi_1^0$ set into the set of irreducible (resp.~prime) elements while maintaining control over the set of prime (reps.~irreducible) elements.  Moreover, our integral domains will extend $\mathbb{Z}$ and we will perform our noncomputable coding into the normal integer primes as in \cite{JoeDamir}.

\section{Strongly Computable Finite Factorization Domains}

In Section 3, we will build a computable integral domain $A$ such that the set of irreducible elements of $A$ is computable but the set of prime elements of $A$ is not computable.  The idea is that we will turn off the primeness of a normal integer prime $p_i$ in response to a $\Sigma_1^0$ event (such as program $i$ halting) by introducing a new element $x$ with $p_i \mid x^2$ but $p_i \nmid x$.  In doing this, we will expand $A$ and we will want to ensure that we can compute the irreducible elements in the resulting integral domain.  Since we are adding a new element, this construction will be analogous to expanding our original $A$ to the polynomial ring $A[x]$.  However, there is a potential problem here in that even if the irreducible elements of an integral domain $A$ are computable, it need not be the case the the irreducible elements of $A[x]$ are computable.  In fact, as mentioned in the introduction, there are computable fields $F$ (where the irreducibles are trivially computable because no element is irreducible) such that the irreducibles of $F[x]$ are not computable.

To remedy this situation, we will ensure that the integral domains in our construction have a stronger property.  As motivation, we first summarize Kronecker's method for finding the divisors of an element $\mathbb{Z}[x]$, and hence for determining whether an element is irreducible.  Let $f(x) \in \mathbb{Z}[x]$ be nonzero, and let $n = \deg(f(x))$.  We try to restrict the set of possible divisors to a finite set that we need to check.  Since the degree function is additive, notice that any divisor of $f(x)$ has degree at most $n$.  Now perform the following:
\begin{itemize}
\item Notice that if $g(x) \in \mathbb{Z}[x]$ and $g(x) \mid f(x)$ in $\mathbb{Z}[x]$, then $g(a) \mid f(a)$ for all $a \in \mathbb{Z}$.
\item Find $n+1$ many points $a \in \mathbb{Z}$ with $f(a) \neq 0$ (which exist because $f(x)$ has at most $n$ roots).  Notice that each such $f(a)$ has only finitely many divisors in $\mathbb{Z}$.
\item For each of the possible choices of the divisors of these values in $\mathbb{Z}$, find the unique interpolating polynomial in $\mathbb{Q}[x]$ of degree at most $n$.
\item Check if any of these polynomials are in $\mathbb{Z}[x]$, and if so, check if they divide $f(x)$ in $\mathbb{Z}[x]$.
\item Compile the resulting list of divisors.
\end{itemize}
Therefore, we can compute the finite set of divisors of any element of $\mathbb{Z}[x]$.  Since we know the units of $\mathbb{Z}[x]$, it follows that we can computably determine if an element of $\mathbb{Z}[x]$ is irreducible.

The key algebraic fact that makes Kronecker's method work is that every nonzero element of $\mathbb{Z}$ has only finitely many divisors.  Integral domains with this property were defined and studied in \cite{AAZ-Factor1,AAZ-Factor2,AndersonMullins}.

\begin{definition}
Let $A$ be an integral domain.
\begin{itemize}
\item $A$ is a {\em finite factorization domain}, or FFD, if every nonzero element has only finitely many divisors up to associates.
\item $A$ is a {\em strong finite factorization domain} if every nonzero element has only finitely many divisors.
\end{itemize}
\end{definition}

We now define an effective analogue of strong finite factorization domains.  In addition to wanting our ring to be computable, we also want the stronger property that we can compute the finite set of divisors of any nonzero element.  Instead of using the word ``strong" twice, we adopt the following definition.

\begin{definition}
A {\em strongly computable finite factorization domain}, or SCFFD, is a computable integral domain $A$ equipped with a computable function $D$ such that for all $a \in A \backslash \{0\}$, we have that $D(a)$ is (a canonical index for) the finite set of divisors of $a$ in $A$.
\end{definition}

\begin{proposition}
Let $A$ be an SCFFD equipped with divisor function $D$.
\begin{enumerate}
\item The set $U(A)$ is a finite set that can be computed from $A$.
\item The set of irreducible elements of $A$ is computable.
\end{enumerate}
\end{proposition}

\begin{proof}
For the first claim, simply notice that $U(A) = D(1)$.  For the second, given any $a \in A$, we have that $a$ is irreducible if and only it nonzero, not a unit, and its only divisors are units and associates.  Suppose then that we are given an arbitrary $a \in A$.  We can check whether $a$ is zero or a unit (by part 1), and if either is true, then $a$ is not irreducible.  Otherwise, then since $a \neq 0$, we can compute the finite set $D(a)$ of divisors of $a$.  Since we can also compute the finite set $U(A)$, we can examine each $b \in D(a)$ in turn to determine whether $b \in U(A)$ or whether there exists $u \in U(A)$ with $b = au$.  If this is true for all $b \in D(a)$, then $a$ is irreducible in $A$, and otherwise it is not.
\end{proof}

If we include an additional assumption that $A$ is a UFD, then we have a converse to the previous result.

\begin{proposition}
Let $A$ be a computable integral domain with the following properties:
\begin{itemize}
\item $A$ is a UFD.
\item $U(A)$ is finite.
\item The set of irreducible elements of $A$ is computable.
\end{itemize}
We can then equip $A$ with a computable function $D$ so that $A$ becomes an SCFFD.
\end{proposition}

\begin{proof}
We first argue that we can computably factor elements of $A$ into irreducibles.  Let $a \in A$ be nonzero and not a unit.  Since the set of irreducibles of $A$ is computable, we can check whether $a$ is irreducible.  If not, we search until we find two nonzero nonunit elements of $A$ whose product is $a$.  We can now check if these factors are irreducible, and if not we can repeat to factor them.  Notice that this process must eventually produce finitely many irreducibles whose product is $a$ by K\"{o}nig's Lemma together with the fact that there are no infinite descending chains of strict divisibilities in a UFD.

We now define our function $D$.  Let $a \in A \backslash \{0\}$ be arbitrary.  Check if $a \in U(A)$ (which is possible because $U(A)$ is finite and computable from $A$), and if so, define $D(a)$ to equal $U(A)$.  If $a \notin U(A)$, then we we can computably factor it into irreducibles $q_i$ so that $a = q_1q_2 \dots q_n$.  Since $U(A)$ is finite, we can now computably check if any of the $q_i$ are associates of each other, and if so we can find witnessing units.  Thus, we can write $a = up_1^{k_1} \cdots p_m^{k_m}$ where $w \in U(A)$, each $p_i$ is irreducible, each $k_i \in \mathbb{N}^+$, and $p_i$ and $p_j$ are not associates whenever $i \neq j$.  Since $A$ is a UFD, we then have that the set of divisors of $a$ equals the set of elements of the form $wp_1^{\ell_1} \cdots p_m^{\ell_m}$ where $w \in U(A)$ and $0 \leq \ell_i \leq k_i$ for all $i$.  Thus, we can define $D(a)$ to be this finite set.
\end{proof}

In contrast, there are SCFFDs that are not UFDs, such as $\mathbb{Z}[\sqrt{-5}]$.  More generally, the ring of integers in any imaginary quadratic number field is an SCFFD.  To see this, Let $K$ be an imaginary quadratic number field, and fix an integral basis of $\mathcal{O}_K$.  Using this integral basis, we can view $\mathcal{O}_K$ as a computable integral domain in such a way that the norm function and divisibility relation are both computable on $\mathcal{O}_K$ (see \cite[Proposition 1.4]{JoeDamir}).  Given any $n \in \mathbb{N}$, there are only finitely many elements of norm $n$, and moreover we can compute the finite set of such elements.  Now given any nonzero $a \in A$, we can compute $N(a)$, examine all elements of norm dividing $N(a)$, and check which of them divide $a$ (since the divisibility relation is computable) to compute the set of divisors of $a$.

Let $A$ be a computable integral domain and let $F$ be the field of fractions of $A$.  Recall that elements of $F$ are equivalence classes of pairs of elements of $A$.  If we were to allow multiple representations of elements, we can of course work with pairs of elements of $A$ and define addition and multiplication on these elements computably.  Nonetheless, a computable ring is defined in a way that forbids such multiple representations, so it is not immediately obvious that we can view $F$ as a computable field.  However, since a computable integral domain is coded as a subset of $\mathbb{N}$, we can view pairs of elements $(a,b) \in A^2$ with $b \neq 0$ as being coded by elements of $\mathbb{N}^2$, which in turn can be coded by elements of $\mathbb{N}$.  Thus, we can view the field of fractions $F$ as a computable field by working only with pairs $(a,b)$ such that there is no strictly smaller pair $(c,d)$ in the usual ordering of $\mathbb{N}$ with $ad = bc$.  In this way, we can still define addition and multiplication computably be searching back for the smallest equivalent representative.

In general, for a computable integral domain $A$, it may not be possible to build the field of fractions as a computable extension of $A$, because it may not be possible to determine when an element $\frac{a}{b} \in F$ is actually an element of $A$.  The issue is that we may not be able to determine if $b \mid a$ because the divisibility relation may not be computable.  However, we have the following.

\begin{corollary}
If $A$ is an SCFFD, then the field of fractions of $A$ is a computable field, and we can computably build it as an extension of $A$.
\end{corollary}

\begin{proof}
Notice that in the field of fractions of $A$, we have that $\frac{a}{b} \in A$ if and only if $b \mid a$, which is if and only if $b \in D(a)$.  Now since $A$ is a computable integral domain, it is coded as a subset of $\mathbb{N}$.  We can now add on minimal pairs $(a,b)$ such that $b \nmid a$.  With this, we can define addition and multiplication
\end{proof}

In fact, we can computably ``reduce" fractions over an SCFFD to lowest terms, as we now show.

\begin{proposition} \label{p:ReduceFractionsOverSCFFD}
Let $A$ be an SCFFD and let $F$ be the field of fractions of $A$.  Given an arbitrary pair of elements $a,b \in R$ with $b \neq 0$, we can computably find a pair of elements $c,d \in R$ with $d \neq 0$, with $\frac{c}{d} = \frac{a}{b}$ in $F$, and such that the only common divisors of $c$ and $d$ are the units of $A$.
\end{proposition}

\begin{proof}
First notice that if $a = 0$, then we may take $c = 0$ and $d = 1$.  Suppose then that $a \neq 0$.  Since we also have that $b \neq 0$, we can now computably determine the finite set of divisors of each of $a$ and $b$, and thus can computably build the finite set $S$ of common divisors of $a$ and $b$, i.e.~$S = D(a) \cap D(b)$.  For each $r \in S$, we can computably determine the number $|\{s \in S : s \mid r\}| = |D(r) \cap S|$.  Fix an $r \in S$ such that $|\{s \in S : s \mid r\}|$ is as large as possible.  Since $r$ is a common divisor of $a$ and $b$, we can computably search for $c,d \in A$ such that $rc = a$ and $rd = b$.  Notice that $d \neq 0$ (because $b \neq 0$) and $\frac{a}{b}= \frac{c}{d}$.  Suppose now that $t$ is a common divisor of $c$ and $d$.  We then have that $rt$ is a common divisor of $a$ and $b$, so $rt \in S$.  By definition of $R$, this implies that $|\{s \in S : s \mid rt\}| \leq |\{s \in S : s \mid r\}|$.  Since $\{s \in S : s \mid r\} \subseteq \{s \in S : s \mid rt\}$, it follows that $\{s \in S : s \mid r\} \subseteq \{s \in S : s \mid rt\}$.  Thus ,$|\{s \in S : s \mid rt\}| = |\{s \in S : s \mid r\}|$.  In particular, we must have $rt \mid r$, so $t \in U(A)$.
\end{proof}

Notice this reduction need not be unique, even up to units.  In the SCFFD $\mathbb{Z}[\sqrt{-5}]$ we have that
\[
\frac{2}{1+\sqrt{-5}} = \frac{1-\sqrt{-5}}{3}
\]
where there are no nonunit common factors for the numerator and denominator of either side.

By \cite[Proposition 5.3]{AAZ-Factor1} and \cite[Theorem 5]{AndersonMullins}, if $A$ is a (strong) finite factorization domain, then so is $A[x]$.  We now prove an effective analogue of this result.  Notice first that if $A$ is a finite integral domain, then $A$ is a finite field, and $A[x]$ is trivially an SCFFD because given $f(x) \in A[x] \backslash \{0\}$, every divisor $g(x)$ of $f(x)$ must satisfy $\deg(g(x)) \leq \deg(f(x))$, and so we need only check each of the finitely many possibilities (which is possible because we can computably search for quotients and remainders).  We now handle the infinite case.

\begin{theorem} \label{t:PolynomialRingOverSCFFDisSCFFD}
If $A$ is an infinite SCFFD, then so is $A[x]$.  Moreover, given an index for a function $D$ witnessing that $A$ is an SCFFD, we can computably obtain an index for a function $D'$ extending $D$ to witness the fact that $A[x]$ is an SCFFD.
\end{theorem}

Before jumping into the proof, we give two lemmas.

\begin{lemma} \label{l:InterpoteAndCheckIfInA}
Let $A$ be an SCFFD, let $n \in \mathbb{N}^+$, let $a_0, a_1, \dots, a_n \in A$ be distinct and let $b_0, b_1, \dots, b_n \in R$.  Let $F$ be the field of fractions of $A$.  There is exactly one polynomial $p(x) \in F[x]$ of degree at most $n$ with $p(a_i) = b_i$ for all $i$.  Furthermore, we can computably construct $p(x)$ in $F[x]$, and can computably determine if $p(x) \in A[x]$. \\
\end{lemma}

\begin{proof}
Uniqueness follows from that fact that if two polynomials over a field having degree at most $n$ agree at $n+1$ points, then they must be the same polynomial.  For existence, using Lagrange's method of interpolation for $n+1$ distinct points of the form $(a_i, b_i)$ will result in a polynomial of the following form:
\[
p(x) = \sum_{i=0}^{n} b_i \cdot \frac{(x-a_0) \cdots (x-a_{i-1}) (x-a_{i+1}) \cdots (x-a_n)}{(a_i-a_0) \cdots (a_i-a_{i-1}) (a_i-a_{i+1})\cdots (a_i-a_n)}
\]
Notice that the denominator is nonzero because $A$ is an integral domain and $a_i \neq a_j$ whenever $i \neq j$.  We can computably expand $p(x)$ to write it as $p(x) = \sum_{i=0}^n \frac{c_i}{d_i} x^i$.  We then have that $p(x) \in A[x]$ if and only if $d_i \mid c_i$ for all $i$, which we can verify by checking if $d_i \in D(c_i)$ for all $i$.
\end{proof}

\begin{lemma} \label{l:DivisibilityRelationOnPolyRingIsComputable}
Suppose that $A$ is an SCFFD.  The divisibility relation on $A[x]$ is computable, i.e.~given $f(x),g(x) \in A[x]$, we can computably determine if $f(x) \mid g(x)$ in $A[x]$.
\end{lemma}

\begin{proof}
Let $f(x),g(x) \in A[x]$ be arbitrary.  If $g(x) = 0$, then trivially we have $f(x) \mid g(x)$.  Suppose then that both $g(x)$ is nonzero.  Perform polynomial long division (or search) to find $q(x),r(x) \in F[x]$ with $f(x) = q(x)g(x) + r(x)$ and either $r(x) = 0$ or $\deg(r(x)) < \deg(g(x))$.  Since quotients and remainders are unique in $F[x]$, we have that $g(x) \mid f(x)$ in $A[x]$ if and only if $q(x) \in A[x]$ and $r(x) = 0$.  Since we can computably determine if an element of $F[x]$ is in $A[x]$ as in Lemma \ref{l:InterpoteAndCheckIfInA}, this completes the proof.
\end{proof}

\begin{proof}[Proof of Theorem \ref{t:PolynomialRingOverSCFFDisSCFFD}]
Let $f(x) \in A[x]$ be arbitrary, and let $n = \deg(f(x))$.  Suppose that $g(x) \in A[x]$ is such that $g(x) \mid f(x)$.  First notice that $\deg(g(x)) \leq n$ because the degree function is additive (as $A$ is an integral domain).  Now if we fix $h(x) \in A[x]$ with $g(x)h(x) = f(x)$, we then have $g(a)h(a) = f(a)$ for all $a \in A$, so since $f(a),g(a),h(a) \in A$ for all $a \in A$, we have that $g(a) \mid f(a)$ for all $a \in A$.

Search until we find $n+1$ many distinct elements $a_0,a_1,\dots,a_n \in A$ such that $f(a_i) \neq 0$ for all $i$ (such $a_i$ exist because $A$ is infinite and $f(x)$ has at most $n$ roots in $A$).  Since $A$ is an SCFFD, we have that $f(a_i)$ has only finitely many divisors for each $i$, and we can compute the finite sets $D(f(a_i))$.  Suppose that we pick elements $b_i \in D(f(a_i))$ for each $i$.  From Lemma \ref{l:InterpoteAndCheckIfInA}, there is a unique element $p(x) \in F[x]$ with $\deg(p(x)) \leq n$ and $p(a_i) = b_i$ for all $i$, and we can compute this polynomial $p(x)$ and determine if $p(x) \in A[x]$.  As we do this for each choice of the $b_i$, we obtain a finite subset of $A[x]$ of all possible divisors of $f(x)$.  Now using Lemma \ref{l:DivisibilityRelationOnPolyRingIsComputable}, we can thin out this set to form the actual finite set of divisors of $f(x)$.
\end{proof}

\section{Irreducibles Computable and Primes Noncomputable}

Let $A$ be an integral domain that is an SCFFD and suppose that $q$ is a prime of $A$.  Suppose that we want to destroy the primeness of $q$ while maintaining its irreducibility (say in response to a $\Sigma_1^0$ event such as the halting of a program).  The idea is to introduce a new element $x$ so that $q \mid x^2$ but $q \nmid x$.  If we let $F$ be the field of fractions of $A$, then we can accomplish this by working in $F[x]$, and extending $A$ to the subring $A[\frac{x^2}{q}]$ of $F[x]$.  More explicitly, $A[\frac{x^2}{q}]$ is the set of all polynomials of the form
\[
a_0 + a_1x + \frac{a_2}{q} \cdot x^2 + \frac{a_3}{q} \cdot x^3 + \frac{a_4}{q^2} \cdot x^4 + \frac{a_5}{q^2} \cdot x^5 + \dots + \frac{a_n}{q^{\lfloor n/2 \rfloor}} \cdot x^n
\]
where each $a_i \in A$.  Although this works, we will find it more convenient notationally to work with subring $B$ of $F[x]$ consisting of those polynomials of the form
\[
a_0 + a_1x + \frac{a_2}{q^2} \cdot x^2 + \frac{a_3}{q^3} \cdot x^3 + \frac{a_4}{q^4} \cdot x^4 + \frac{a_5}{q^5} \cdot x^5 + \dots + \frac{a_n}{q^n} \cdot x^n
\]
where each $a_i \in A$.

\begin{theorem} \label{t:PropoertiesAfterDestroyingOnePrime}
Let $A$ be an SCFFD and let $q \in A$ be prime.  Let $F$ be the field of fractions of $A$ and let $B$ be the subring of $F[x]$ consisting of those polynomials of the form
\[
a_0 + a_1x + \frac{a_2}{q^2} \cdot x^2 + \frac{a_3}{q^3} \cdot x^3 + \frac{a_4}{q^4} \cdot x^4 + \frac{a_5}{q^5} \cdot x^5 + \dots + \frac{a_n}{q^n} \cdot x^n
\]
where each $a_i \in A$.  We then have the following.
\begin{enumerate}
\item \label{e:DivisorsPreservedInPrimeDestruction} For any $a \in A$, the set of divisors of $a$ in $A$ equals the set of divisors of $a$ in $B$.
\item $B$ is an SCFFD.  Moreover, given $A$, $q$, and an index for a function $D$ witnessing that $A$ is an SCFFD, we can computably build $B$ as an extension of $A$ and obtain an index for a function $D'$ witnessing that $B$ is an SCFFD with the property that $D'(a) = D(a)$ for all $a \in A$.
\item \label{e:UnitsPreservedInPrimeDestruction} $U(B) = U(A)$.
\item If $p$ is irreducible in $A$, then $p$ is irreducible in $B$.
\item If $p_1,p_2 \in A$ are irreducibles that are not associates in $A$, then they are not associates in $B$.
\item $q$ is not prime in $B$.
\item If $p$ is a prime of $A$ that is not an associate of $q$, then $p$ is prime in $B$.
\end{enumerate}
\end{theorem}

\begin{proof}
\begin{enumerate}
\item Let $a \in A$.  Clearly, if an element of $A$ divides $a$ in $A$, then it divides $a$ in $B$.  For the converse, since the degree function is additive on $F[x]$, if $f(x),g(x) \in B$ are such that $a = f(x)g(x)$, then we must have $\deg(f(x)) = 0 = \deg(g(x))$, and hence $f(x),g(x) \in A$.

\item Notice first that we computably build $B$ as an extension of $A$ trivially, because if $\frac{a}{q^k} = \frac{b}{q^k}$, then $a = b$ (so there is no issue of distinct representations).  The proof that $B$ is an SCFFD is analogous to the proof of Theorem \ref{t:PolynomialRingOverSCFFDisSCFFD}, with a few straightforward modifications.  Given $f(x) \in B$ with $\deg(f(x)) = n$, to determine the divisors of $f(x)$ in $B$, we note the following:
\begin{itemize}
\item Notice that if $f(x) \in B$ and $a \in A$, then in general it need not be the case that $f(a) \in A$.  However, we will only plug in values $q^i$ for $i \geq n$ to avoid this issue.  Suppose then that $g(x) \in B$ with $g(x) \mid f(x)$, and fix $h(x) \in B$ with $g(x)h(x) = f(x)$.  We then have that $\deg(g(x)) \leq n$ and $\deg(h(x)) \leq n$.  Thus, for any $i \geq n$, we have $f(q^i), g(q^i), h(q^i) \in A$, and so $g(q^i) \mid f(q^i)$ in $A$.  Since there are infinitely many $i \geq n$, and these $q^i$ provide an infinite supply of distinct elements (because $A$ is an integral domain), we can plug in $n+1$ many such values with $f(q^i) \neq 0$ to form the basis for our Lagrange interpolations. 
\item We can computably determine if an element $p(x) \in F[x]$ is actually an element of $B$.  The key question is given $a,b \in A$ with $b \neq 0$ and a $k \geq 2$, can we determine if we can write an element $\frac{a}{b}$ of $F$ in the form $\frac{c}{q^k}$.  Notice that this is possible if and only if there exists $c \in A$ with $aq^k = bc$, which is if and only if $b \mid aq^k$.  Since $A$ is an SCFFD, we can computably determine if $b \in D(aq^k)$, and furthermore in this case we can computably find $c$ with $bc = aq^k$ and hence $\frac{a}{b}= \frac{c}{q^k}$,   Thus, we can determine if an element of $F[x]$ is an element of $B$, and if so write it in the above form.
\item The divisibility relation is computable on $B$ as in Lemma \ref{l:DivisibilityRelationOnPolyRingIsComputable}, because we can computably determine if an element of $F[x]$ is an element of $B$ as just mentioned.
\end{itemize}
This shows that $B$ is an SCFFD and allows us to compute $D'$ uniformly from $A$ and $D$.  Finally, notice that $D'$ extends $D$ by \ref{e:DivisorsPreservedInPrimeDestruction}.

\item Immediate from \ref{e:DivisorsPreservedInPrimeDestruction} and the fact that $U(B) = D(1)$.

\item This follow from \ref{e:DivisorsPreservedInPrimeDestruction} and \ref{e:UnitsPreservedInPrimeDestruction}.

\item Immediate from \ref{e:UnitsPreservedInPrimeDestruction}.

\item Notice that $q$ is nonzero and not a unit by \ref{e:UnitsPreservedInPrimeDestruction}.  We have that $q \mid x^2$ in $B$ because $\frac{1}{q} \cdot x^2 = \frac{q}{q^2} \cdot x^2 \in B$, but $q \nmid x$ because $\frac{1}{q} \cdot x \notin B$ as $q$ is not a unit (and this is the only possible witness for divisibility because $F[x]$ is an integral domain).  Therefore, $q$ is not prime in $B$.

\item Let $p$ be a prime of $A$ that is not an associate of $q$.  Notice that $p$ is nonzero and not a unit of $B$ by \ref{e:UnitsPreservedInPrimeDestruction}.  Let $f(x),g(x) \in B$, and suppose that $p \mid f(x)g(x)$ in $B$.  Write out
\begin{align*}
f(x) & = a_0 + a_1x + \frac{a_2}{q^2} \cdot x^2 + \frac{a_3}{q^3} \cdot x^3 + \dots + \frac{a_n}{q^n} \cdot x^n \\
g(x) & = b_0 + b_1x + \frac{b_2}{q^2} \cdot x^2 + \frac{b_3}{q^3} \cdot x^3 + \dots + \frac{b_n}{q^n} \cdot x^n \\
f(x)g(x) & = c_0 + c_1x + \frac{c_2}{q^2} \cdot x^2 + \frac{c_3}{q^3} \cdot x^3 + \dots + \frac{c_n}{q^n} \cdot x^n
\end{align*}
Since $p \mid f(x)g(x)$ in $B$, we have that $p \mid c_i$ in $A$ for all $i$.  Assume that $p \nmid f(x)$ and $p \nmid g(x)$ in $B$.  Then there must exist $i$ and $j$ such that $p \nmid a_i$ in $A$ and $p \nmid b_j$ in $A$.  Let $k$ and $\ell$ be largest possible such that $p \nmid a_k$ in $A$ and $p \nmid b_{\ell}$ in $A$.  Now element $c_{k+\ell}$ will be a sum of terms, one of which will be $a_kb_{\ell}q^j$ for some $j \in \{0,1\}$, while other terms will be divisible by $p$ in $A$.  Since $p$ divides $c_{k+\ell}$, it follows that $p \mid a_kb_{\ell}q^j$ in $A$.  However, this is a contradiction because $p$ is prime in $A$ but divides none of $a_k$, $b_{\ell}$, or $q$ (the last because $p$ is not an associate of $q$ in $A$).
\end{enumerate}
\end{proof}

We now show that we can code an arbitrary $\Pi_1^0$ set into the primes of an integral domain $A$ while maintaining the computability of the irreducible elements.  In fact, we perform our coding within the normal integer primes and can make the resulting integral domain an SCFFD.

\begin{theorem} \label{t:Pi1ControlOfPrimes}
Let $S$ be a $\Sigma_1^0$ set, and let $p_0,p_1,p_2,\dots$ list the usual primes from $\mathbb{N}$ in increasing order.  There exists an SCFFD $A$ such that:
\begin{itemize}
\item $\mathbb{Z}$ is a subring of $A$.
\item $U(A) = \{1,-1\}$.
\item Every $p_i$ is irreducible in $A$.
\item $p_i$ is prime in $A$ if and only if $i \notin S$.
\end{itemize}
\end{theorem}

\begin{proof}
If $S = \emptyset$, this is trivial by letting $A = \mathbb{Z}$.  Assume then that $S \neq \emptyset$.  If $S$ is finite, say $|S| = n$, then we can trivially fix a computable injective function $\alpha \colon \{1,2,\dots,n\} \to \mathbb{N}$ with $\text{range}(\alpha) = S$.  If $S$ is infinite, then we can fix a computable injective function $\alpha \colon \mathbb{N} \to \mathbb{N}$ with $\text{range}(\alpha) = S$ by Proposition \ref{p:Sigma1IffRangeComputable}.

We build our computable SCFFD $A$ in stages, starting by letting $A_0 = \mathbb{Z}$ and letting $D_0(a)$ be the finite set of divisors of $a$ for all $a \in \mathbb{Z} \backslash \{0\}$.  Suppose that we are at a stage $k$ and have constructed an SCFFD $A_k$ together with witnessing function $D_k$.  We now extend $A_k$ to $A_{k+1}$ by destroying the primality of $p_{\alpha(k)}$ as in the construction of Theorem \ref{t:PropoertiesAfterDestroyingOnePrime} using a new indeterminate $x_k$.  In other words, letting $F_k$ be the field of fractions of $A_k$, we let $A_{k+1}$ be the subring of $F_k[x]$  consisting of those polynomials of the form
\[
a_0 + a_1x + \frac{a_2}{p_{\alpha(k)}^2} \cdot x_k^2 + \frac{a_3}{p_{\alpha(k)}^3} \cdot x_k^3 + \frac{a_4}{p_{\alpha(k)}^4} \cdot x_k^4 + \dots + \frac{a_n}{p_{\alpha(k)}^n} \cdot x_k^n
\]
where each $a_i \in A_k$.  We continue this process through the construction of $A_n$ if $|S| = n$, and infinitely often if $S$ is infinite.  Using Theorem \ref{t:PropoertiesAfterDestroyingOnePrime}, the following properties hold by induction on $k$:
\begin{itemize}
\item $A_k$ is an SCFFD with witnessing function $D_k$ extending $D_i$ for all $i < k$.
\item $U(A_k) = \{1,-1\}$.
\item Every $p_i$ is irreducible in $A_k$.
\item $p_i$ is prime in $A_k$ if and only if $i \notin \{\alpha(1),\alpha(2),\dots,\alpha(k)\}$.
\end{itemize}
Now if $S$ is finite, say $|S| = n$, then it follows that the integral domain $A_n$ has the required properties.

Suppose then that $S$ is infinite, and let $A = A_{\infty} = \bigcup_{k=0}^{\infty} A_k$.  Also, let $D = \bigcup_{k=1}^{\infty} D_k$, which makes sense because the $D_i$ extend each other as functions.  Notice that $D$ is a computable function and that for any $a \in A_k$, we have that the set of divisors of $a$ in $A$ equals the set of divisors of $a$ in $A_k$, so $D(a) = D_k(a)$ is the finite set of divisors of $a$ in $A$.  Therefore, $A$ is an SCFFD as witnessed by $D$.  Since $U(A_k) = \{1,-1\}$ for all $k \in \mathbb{N}$, it follows that $U(A) = \{1,-1\}$.  Since we maintain the units and divisibility at each stage, it also follows that every $p_i$ is irreducible in $A$.

We now show that $p_i$ is prime in $A$ if and only if $i \notin S$.  First notice that each $p_i$ is nonzero and not a unit of $A$.
\begin{itemize}
\item Suppose first that $i \notin S$.  We then have that $i \notin \text{range}(\alpha)$, so $p_i$ is prime in every $A_k$ by the last property above.  Let $a,b \in A$, and suppose that $p_i \mid ab$ in $A$.  Fix $c \in A$ with $p_ic = ab$.  Go to a point $k$ where each of $p_i,a,b,c$ exist.  We then have that $p_i \mid ab$ in $A_k$, so as $p_i$ is prime in $A_k$, either $p_i \mid a$ in $A_k$ or $p_i \mid b$ in $A_k$.  Therefore, either $p_i \mid a$ in $A$ or $p_i \mid b$ in $A$.  It follows that $p_i$ is prime in $A_k$.
\item Suppose now that $i \in S$.  Thus, we can fix $k \in \mathbb{N}$ with $\alpha(k) = i$.  We then have that $p_i$ is not prime in $A_{k+1}$ by the last property above.  Fix $a,b \in A_{k+1}$ such that $p_i \mid ab$ in $A_{k+1}$ but $p_i \nmid a$ in $A_{k+1}$ and $p_i \nmid b$ in $A_{k+1}$.  Since the $D_i$ extend each other as functions, and $A$ is an SCFFD as witnessed by $D$, it follow that $p_i \mid ab$ in $A$ but $p_i \nmid a$ in $A$ and $p_i \nmid b$ in $A$.  Therefore, $p_i$ is not prime in $A$.
\end{itemize}
\end{proof}

\begin{corollary}
There exists a computable integral domain $A$ such that the set of irreducible elements of $A$ is computable but the set of prime elements of $A$ is not computable.
\end{corollary}

\begin{proof}
Fix a noncomputable $\Sigma_1^0$ set $S$, and let $A$ be the SCFFD given by Theorem \ref{t:Pi1ControlOfPrimes}.  Since $A$ is an SCFFD, it is a computable integral domain and the set of irreducible elements of $A$ is computable.  However, the set of prime elements of $A$ is not computable, because if we could compute it, then we could compute $S$, which is a contradiction.
\end{proof}

\section{Primes Computable and Irreducibles Noncomputable}

Consider the subring $A = \mathbb{Z} + x\mathbb{Z} + x^2\mathbb{Q}[x]$ of $\mathbb{Q}[x]$.  In other words, $A$ is the set of polynomials of the form $q_0 + q_1x + q_2x^2 + \dots + q_nx^n$ where $q_0 \in \mathbb{Z}$ and $q_1 \in \mathbb{Z}$.  As mentioned in the introduction, each normal integer prime is irreducible in $A$ but is not prime in $A$.  It is also a standard fact for $p(x) \in A$, we have that $p(x)$ is prime in $A$ if and only if $p(x)$ is irreducible in $\mathbb{Q}[x]$ and $p(0) \in \{1,-1\}$.

We will generalize this construction by replacing $\mathbb{Z}$ with an arbitrary integral domain.  Suppose that $R$ is an integral domain, and let $F$ be its field of fractions.  Consider the subring $A = R + xR + x^2F[x]$ of $F[x]$, i.e.~$A$ is the set of polynomials of the form $q_0 + q_1x + q_2x^2 + \dots + q_nx^n$ where $q_0 \in R$ and $q_1 \in R$.  Such an integral domain $A$ is particularly nice from our perspective because the irreducibles in $R$ will remain irreducible in $A$ (so all of the complexity of irreducibles remain), but no element of $R$ is prime in $A$ (so any complexity of primes is ``erased").  Moreover, we can reduce the complexity of primality of elements of $A$ to that of irreducibles in the polynomial ring over a field, about which a great deal is understood.

\begin{lemma} \label{l:PrimesInAAreNonconstantAndIrreducible}
Let $R$ be an integral domain with field of fractions $F$.  Consider the subring $A = R + xR + x^2F[x]$ of $F[x]$.  Let $p(x) \in A$.  If $p(x)$ is prime in $A$, then $p(x)$ is non-constant and irreducible in $F[x]$.
\end{lemma}

\begin{proof}
We prove the contrapositive, i.e.~if $p(x) \in A$ is either constant or not irreducible, then $p(x)$ is not prime in $A$.

Suppose first that $p(x)$ is a constant, and fix $k \in R$ with $p(x) = k$.  If $k \in \{0\} \cup U(R)$, then $k$ is either zero or a unit, so $k$ is not prime in $A$ by definition.  Suppose then that $k \notin \{0\} \cup U(R)$.  Notice that $k \mid x^2$ in $A$ because $\frac{1}{k} \cdot x^2 \in A$, but $k \nmid x$ in $A$ because $\frac{1}{k} \cdot x \notin A$.  Therefore, $p(x) = k$ is not prime in $A$.

Suppose now that $p(x) \in A$ is non-constant and not irreducible in $F[x]$.  Since $p(x)$ is non-constant, it is not a unit in $F[x]$.  Fix $g(x),h(x) \in F[x]$ with $p(x) = g(x)h(x)$ and such that $0 < \deg(g(x)) < \deg(p(x))$ and $0 < \deg(h(x)) < \deg(p(x))$.  Now since $g(x),h(x) \in F[x]$, the constant terms and coefficients of $x$ in these polynomials need not be in $R$.  Let $b$ be the product of the denominators of these coefficients in $g(x)$, and let $c$ be the product of the denominators of these coefficients in $h(x)$.  We then have that $p(x) \cdot bc = (b \cdot g(x)) \cdot (c \cdot h(x))$ where both $b \cdot g(x) \in A$ and $c \cdot h(x) \in A$.  Since $bc \in R \subseteq A$, we have that $p(x) \mid (b \cdot g(x)) \cdot (c \cdot h(x))$ in $A$.  However, notice that $p(x) \nmid b \cdot g(x)$ in $A$ because $\deg(b \cdot g(x)) < \deg(p(x))$ and $p(x) \nmid c \cdot h(x)$ because $\deg(c \cdot h(x)) < \deg(p(x))$.  Therefore, $p(x)$ is not prime in $A$.
\end{proof}

\begin{lemma} \label{l:CharacterizePrimesInA}
Let $R$ be an integral domain with field of fractions $F$.  Consider the subring $A = R + xR + x^2F[x]$ of $F[x]$.  Let $p(x) \in A$ and suppose that $p(x)$ is irreducible in $F[x]$.  The following are equivalent.
\begin{enumerate}
\item $p(x)$ is prime in $A$.
\item For all $f(x) \in F[x]$, if $p(x)f(x) \in A$, then $f(x) \in A$.
\item For all $g(x) \in A$ such that $p(x) \mid g(x)$ in $F[x]$, we have that $p(x) \mid g(x)$ in $A$.
\end{enumerate}
\end{lemma}

\begin{proof}
$(1) \rightarrow (2)$:  Suppose first that $p(x)$ is prime in $A$.  We know that no constants are prime in $A$ from above, so $p(x)$ is non-constant.  Let $f(x) \in F[x]$ be such that $p(x)f(x) \in A$.  We prove that $f(x) \in A$.  Write $f(x)  = q_0 + q_1x + \dots + q_nx^n$ where each $q_i \in F$.  Let $d$ be the product of the denominators of $q_0$ and $q_1$.  Now $d \in R \subseteq A$ and $d \cdot f(x) \in A$, hence $p(x) \mid p(x) \cdot d \cdot f(x)$ in $A$, i.e.~$p(x) \mid d \cdot (p(x)f(x))$ in $A$.  Since $p(x)$ is prime in $A$, either $p(x) \mid d$ in $A$ or $p(x) \mid p(x)f(x)$ in $A$.  The former is impossible because $p(x)$ is non-constant, so we must have that $p(x) \mid p(x)f(x)$ in $A$.  Fix $h(x) \in A$ with $p(x)h(x) = p(x)f(x)$.  Since $F[x]$ is an integral domain, we conclude that $f(x) = h(x) \in A$.

$(2) \rightarrow (3)$:  Immediate.

$(3) \rightarrow (1)$:  Let $g(x),h(x) \in A$ and suppose that $p(x) \mid g(x)h(x)$ in $A$.  Since $A$ is a subring of $F[x]$, we then have that $p(x) \mid g(x)h(x)$ in $F[x]$.  Now $p(x)$ is irreducible in $F[x]$, so since $F[x]$ is a UFD, we know that $p(x)$ is prime in $F[x]$.  Thus, either $p(x) \mid g(x)$ in $F[x]$ or $f(x) \mid h(x)$ in $F[x]$.  Using $(3)$, we conclude that either $p(x) \mid g(x)$ in $A$ or $p(x) \mid h(x)$ in $A$.  Therefore, $p(x)$ is prime in $A$.
\end{proof}

\begin{proposition} \label{p:ClassifyPrimesInSubringOfFX}
Let $R$ be an integral domain that is not a field, and let $F$ be its field of fractions.  Consider the subring $A = R + xR + x^2F[x]$ of $F[x]$.  An element $p(x) \in A$ is prime in $A$ if and only if $p(x)$ is irreducible in $F[x]$ and $p(0) \in U(R)$.
\end{proposition}

\begin{proof}
We first prove that if $p(x) \in A$ does not satisfy $p(0) \notin U(R)$, then $p(x)$ is not prime in $A$.  If $p(0) = 0$, then fixing any nonzero nonunit $b \in R$ (which exists because $R$ is not a field), we have $p(x) \cdot \frac{x}{b} \in A$ but $\frac{x}{b} \notin A$, so $p(x)$ is not prime in $A$ by Lemma \ref{l:CharacterizePrimesInA}.  Suppose then that $p(0) \notin \{0\} \cup U(R)$.  Write $p(x) = q_nx^n + \dots + q_2x^2 + ax + b$ where $a,b \in R$ and $b \notin \{0\} \cup U(R)$.  We have
\begin{align*}
p(x) \cdot \left(\frac{1}{b} \cdot x\right) & = (q_nx^n + \dots + q_2x^2 + ax + b) \cdot \left(\frac{1}{b} \cdot x\right) \\
& = \left(\frac{q_n}{b}\right) \cdot x^{n+1} + \dots + \left(\frac{q_2}{b}\right) \cdot x^3 + \left(\frac{a}{b}\right) \cdot x^2+x
\end{align*}
Thus, $f(x) \cdot \frac{1}{b} \cdot x \in A$ but $\frac{1}{b} \cdot x \notin A$, so $f(x)$ is not prime in $A$ by Lemma \ref{l:CharacterizePrimesInA}.

We have just shown that $p(x) \in A$ is prime in $A$, then $p(0) \in U(R)$.  We also know that if $p(x) \in A$ is prime in $A$, then $p(x)$ is irreducible in $F[x]$ by Lemma \ref{l:PrimesInAAreNonconstantAndIrreducible}.

Suppose conversely that $p(x)$ is irreducible in $F[x]$ and that $p(0) \in U(R)$.  Using Lemma \ref{l:CharacterizePrimesInA}, to show that $p(x)$ is prime in $A$ it suffices to show that whenever $f(x) \in F[x]$ is such that $p(x)f(x) \in A$, then we must have $f(x) \in A$.  Suppose then that $f(x) \in F[x]$ and $p(x)f(x) \in A$.  Write
\begin{align*}
f(x) & = q_0 + q_1x + q_2x^2 + \dots + q_nx^n \\
p(x) & = a_0 + a_1x + r_2x^2 + \dots + r_nx^n
\end{align*}
where $a_0 \in U(R)$, $a_1 \in R$, each $q_i \in F$, and each $r_i \in F$.  We then have that $p(x)f(x) \in F[x]$ with
\[
p(x)f(x) = q_0a_0 + (q_0a_1 + a_0q_1) x + \dots
\]
As $p(x)f(x) \in A$, we know that $q_0a_0 \in R$ and $q_0a_1 + a_0q_1 \in R$.  Since $q_0a_0 \in R$ and $a_0 \in U(R)$, it follows that $q_0 \in R$.  Using this together with the facts that $a_1 \in R$ and $q_0a_1 + a_0q_1 \in R$, it follows that $a_0q_1 \in R$.  Applying again the fact that $a_0 \in U(R)$, we conclude that $q_1 \in R$.  Since $q_0,q_1 \in R$, it follows that $p(x) \in A$.
\end{proof}

With these results in hand, we now proceed to construct an integral domain $R$ with a complicated set of irreducible elements.  We will want our $R$ to have a ``nice" field of fractions $F$ in the sense that the irreducibles of $F[y]$ will be computable.

\begin{lemma} \label{l:DestorySigma1PrimesLemma}
Let $S$ be a $\Sigma_1^0$ set, and let $p_0,p_1,p_2,\dots$ list the usual primes from $\mathbb{N}$ in increasing order.  There exists a computable UFD $R$ such that:
\begin{itemize}
\item $\mathbb{Z}$ is a subring of $R$, and in fact
\[
\mathbb{Z}[x_1,x_2,\dots] \subseteq R \subseteq \mathbb{Q}(x_1,x_2,\dots),
\]
where there are infinitely many indeterminates if $S$ is infinite, and exactly $n$ of them if $|S| = n$.
\item $U(R) = \{1,-1\}$.
\item $p_i$ is irreducible in $R$ if and only if $i \notin S$.
\end{itemize}
\end{lemma}

\begin{proof}
If $S = \emptyset$, this is trivial by letting $A = \mathbb{Z}$.  Assume then that $S \neq \emptyset$.  If $S$ is finite, say $|S| = n$, then we can trivially fix a computable injective function $\alpha \colon \{1,,2\dots,n\} \to \mathbb{N}$ with $\text{range}(\alpha) = S$.  If $S$ is infinite, then we can fix a computable injective function $\alpha \colon \mathbb{N} \to \mathbb{N}$ with $\text{range}(\alpha) = S$ by Proposition \ref{p:Sigma1IffRangeComputable}.

We build our computable UFD $R$ in stages, starting by letting $R_0 = \mathbb{Z}$.  Suppose that we are at a stage $k$ and have constructed through the integral domain $R_k$.  We now destroy the irreducibility of $p_{\alpha(k)}$ by letting $R_{k+1} = R_k[x_k,\frac{p_{\alpha(k)}}{x_k}]$ as in \cite[Section 3]{JoeDamir}.  We continue this process through the construction of $R_{n+1}$ if $|S| = n$, and infinitely often if $S$ is infinite.  Using \cite[Proposition 3.3 and Theorem 3.10]{JoeDamir}, the following properties hold by induction on $k$:
\begin{itemize}
\item $R_k$ is a Noetherian UFD.
\item $\mathbb{Z}[x_1,x_2,\dots,x_k] \subseteq R_k \subseteq \mathbb{Q}(x_1,x_2,\dots,x_k)$.
\item $U(R_k) = \{1,-1\}$.
\item $p_i$ is irreducible in $R_k$ if and only if $i \notin \{\alpha(1),\alpha(2),\dots,\alpha(k)\}$.
\end{itemize}
Now if $S$ is finite, say $|S| = n$, then it follows that the integral domain $R_n$ has the required properties.

Suppose then that $S$ is infinite, and let $R = R_{\infty} = \bigcup_{k=0}^{\infty} R_k$.   We then have that $R$ has the required properties by the proofs in \cite[Section 4]{JoeDamir} (although they are significantly easier in this case because we never change the units).
\end{proof}

\begin{theorem} \label{t:Pi1ControlOfIrreducibles}
Let $S$ be a $\Sigma_1^0$ set, and let $p_0,p_1,p_2,\dots$ list the usual primes from $\mathbb{N}$ in increasing order.  There exists a computable integral domain $A$ such that:
\begin{itemize}
\item $\mathbb{Z}$ is a subring of $A$.
\item $U(A) = \{1,-1\}$.
\item No $p_i$ is prime in $A$.
\item The set of prime elements of $A$ is computable.
\item $p_i$ is irreducible in $A$ if and only if $i \notin S$.
\end{itemize}
\end{theorem}

\begin{proof}
Let $R$ be the integral domain given by Lemma \ref{l:DestorySigma1PrimesLemma}.  Let $F$ be the field of fractions of $R$.  Since
\[
\mathbb{Z}[x_1,x_2,\dots] \subseteq R \subseteq \mathbb{Q}(x_1,x_2,\dots)
\]
(where there are infinitely many indeterminates if $S$ is infinite, and exactly $n$ of them if $|S| = n$) and the field of fractions of $\mathbb{Z}[x_1,x_2,\dots]$ is $\mathbb{Q}(x_1,x_2,\dots)$, it follows that $F = \mathbb{Q}(x_1,x_2,\dots)$.  Let $A$ be the subring $R + yR + y^2F[y]$ of $F[y]$.  Now we clearly have that $\mathbb{Z}$ is a subring of $A$ and $U(A) = \{1,-1\}$.  Also, each $p_i$ is a constant polynomial in $A$, so is not prime in $A$ by Lemma \ref{l:PrimesInAAreNonconstantAndIrreducible}.  By \cite[Theorem 4.5]{FrohlichShep}, the set of irreducible elements of $F[y]$ is computable, so since $U(R) = \{1,-1\}$, we may use Proposition \ref{p:ClassifyPrimesInSubringOfFX} to conclude that the set of prime elements of $A$ is computable.

Finally, by Lemma \ref{l:DestorySigma1PrimesLemma}, we have that $p_i$ is irreducible in $R$ if and only if $i \notin S$.  Now $R$ is the subring of $A$ consisting of the constant polynomials, so as $U(A) = U(R)$ and divisors of the constant polynomials in $A$ must be constants, it follows that $p_i$ is irreducible in $A$ if and only $p_i$ is irreducible in $R$, which is if and only if $i \notin S$.
\end{proof}

\begin{corollary}
There exists a computable integral domain $A$ such that the set of prime elements of $A$ is computable but the set of irreducible elements of $A$ is not computable.
\end{corollary}

\begin{proof}
Fix a noncomputable $\Sigma_1^0$ set $S$, and let $A$ be the integral domain give by Theorem \ref{t:Pi1ControlOfIrreducibles}.  We then have the set of prime elements of $A$ is computable.  However, the set of irreducible elements of $A$ is not computable, because if we could compute it, then we could compute $S$, which is a contradiction.
\end{proof}


\begin{thebibliography}{99}
\bibitem{AAZ-Factor1}
D. D. Anderson, D. F. Anderson, \and M. Zafrullah,
`Factorization in integral domains',
{\em J. Pure Appl. Algebra}
69(1) (1990) 1--19.
%
\bibitem{AAZ-Factor2}
D. D. Anderson, D. F. Anderson \and M. Zafrullah,
`Factorization in integral domains. II',
{\em J. Algebra}
152(1) (1992) 78--93.
%
\bibitem{AndersonMullins}
D. D. Anderson \and B. Mullins,
`Finite factorization domains',
{\em Proc. Amer. Math. Soc.}
124(2) (1996) 389--396.
%
\bibitem{JoeDamir}
D. Dzhafarov \and J. Mileti,
'The complexity of primes in computable {UFD}s',
{\em to appear}.
%
\bibitem{FrohlichShep}
A. Fr{\"o}hlich \and J. C. Shepherdson,
`Effective procedures in field theory',
{\em Philos. Trans. Roy. Soc. London. Ser. A.}
248 (1956) 407--432.
%
\bibitem{MetakidesNerode}
G. Metakides \and A. Nerode,
`Effective content of field theory',
{\em Ann. Math. Logic}
17(3) (1979) 289--320.
%
\bibitem{MillerNotices}
R. Miller,
`Computable fields and Galois theory',
{\em Notices Amer. Math. Soc.}
55(7) (2008) 798--807.
%
\bibitem{Rabin}
M. Rabin,
`Computable algebra, general theory and theory of computable fields',
{\em Trans. Amer. Math. Soc.}
95 (1960) 341--360.
%
\bibitem{Soare}
R. Soare,
 {\em Recursively enumerable sets and degrees}
 (Springer-Verlag, Berlin, 1987).
%
\bibitem{SHTucker}
V. Stoltenberg-Hansen \and J. V. Tucker,
`Computable rings and fields' in {\em Handbook of computability theory}
 (North-Holland, Amsterdam, 1999).
\end{thebibliography}
\end{document}